\documentclass{amsart}
\usepackage{amssymb}
\usepackage{pstricks}

\title{Alcoved Polytopes II}
\author{Thomas Lam and Alexander Postnikov}
\thanks{T.L. was supported by NSF DMS 0901111, and by a Sloan Fellowship.  A.P. was supported by NSF DMS-1100147 and CAREER DMS-0546209.}
\address{Department of Mathematics, University of Michigan, Ann Arbor, MI, 48109}
\email{tfylam@umich.edu}
\address{Department of Mathematics, M.I.T., Cambridge, MA, 02139}
 \email{apost@math.mit.edu} 
\date{November 25, 2003; modified October, 2005; modified January, 2012}

\newtheorem{theorem}{Theorem}[section]
\newtheorem{proposition}[theorem]{Proposition}
\newtheorem{prop}[theorem]{Proposition}
\newtheorem{corollary}[theorem]{Corollary}
\newtheorem{cor}[theorem]{Corollary}
\newtheorem{definition}[theorem]{Definition}

\newtheorem{remark}[theorem]{Remark}
\newtheorem{lemma}[theorem]{Lemma}
\newtheorem{example}[theorem]{Example}

\newtheorem{question}[theorem]{Question}

\def\inn{\mathrm{in}}
\def\weight{\Lambda}
\def\coweight{\Lambda^{\vee}}
\def\root{L}
\def\coroot{L^{\vee}}
\def\R{\mathbb{R}}
\def\Z{\mathbb{Z}}

\def\<{\left<}
\def\>{\right>}
\def\Waff{W_\mathrm{aff}}
\def\Caff{C_\mathrm{aff}}
\def\Vol{\mathrm{Vol}}
\def\A{\mathcal{A}}
\def\F{\mathcal{F}}
\def\H{\mathcal{H}}

\def\cdes{\mathrm{cdes}}
\def\mod{\,\mathrm{mod}\,}
\def\inv{\mathrm{inv}}
\def\w{\omega}
\def\cmaj{\mathrm{cmaj}}
\def\maj{\mathrm{maj}}
\def\des{\mathrm{des}}

\def\g{\mathcal{G}}
\def\D{\Delta}
\def\P{P}
\def\Q{Q}
\def\sh{\mathrm{sh}}
\def\ll{\lambda}
\def\S{\mathcal{S}}
\def\id{\mathrm{id}}
\def\AA{S}

\newcommand\set[1]{\left\{#1\right\}}

\begin{document}

\begin{abstract}
This is the second of two papers where we study polytopes arising
from affine Coxeter arrangements.  Our results include a formula for
their volumes, and also compatible definitions of hypersimplices,
descent numbers and major index for all Weyl groups.  We give a
$q$-analogue of Weyl's formula for the order of the Weyl group.  For
$A_n$, $C_n$ and $D_4$, we give a Gr\"{o}bner basis which induces
the triangulation of alcoved polytopes.
\end{abstract}

\maketitle

\section{Introduction}
This is the second of two papers where we investigate \emph{alcoved
polytopes} arising from affine Coxeter arrangements.  Let $\Phi \subset V$ be an irreducible crystallographic root system and $W$ be
the corresponding Weyl group. Associated to $\Phi$ is an infinite
hyperplane arrangement known as the affine Coxeter arrangement.
This hyperplane arrangement subdivides $V$ into simplices of the
same volume which are called \emph{alcoves}.  We define a
\emph{proper alcoved polytope} to be a convex polytope $P$ which is
the closure of a union of alcoves.

In~\cite{API}, we studied these polytopes in the special situation of the root
system $\Phi = A_n$.  Two motivating examples for us were the
\emph{hypersimplices} and the \emph{alcoved matroid polytopes}.  Alcoved polytopes 
arising from other root systems have also been studied.  Payne~\cite{Pay} showed that alcoved polytopes with vertices lying in the coweight lattice are normal and Koszul in classical type.  Werner and Yu~\cite{WY} studied generating sets of alcoved polytopes.  Fomin and Zelevinsky's \emph{generalized associahedra}~\cite{FZ} are examples of polytopes which can be realized as alcoved polytopes.

We prove that the volume of an alcoved polytope $P$ is given by
\[
\Vol(P) = \sum_{\bar{w} \in W/C} I(P_{\bar{w}})
\]
where $W/C$ are certain cosets of the Weyl group, $P_{\bar{w}}$ are
certain alcoved polytopes and $I(P)$ denotes the number of integral
coweights lying in $P$.  The group $C \subset W$ was studied
previously by Verma~\cite{Ver}.  The order of the group $C$ is equal
to the index of connection of $\Phi$.  For the case $\Phi =
A_{n-1}$, the group $C$ is the cyclic group generated by the long
cycle $(123\cdots n)$, written in cycle notation.

Recall that the usual hypersimplex $\Delta_{k,n}$ has volume equal
to the Eulerian number $A_{k,n-1}$.  We define \emph{generalized
hypersimplices} $\Delta_k^\Phi$ to be certain alcoved polytopes
which generalize this property of $\Delta_{k,n}$ for each root
system $\Phi$.  To this end we introduce the \emph{circular
descent number} $\cdes: W \rightarrow \Z$ so that the volume of
$\Delta_k^\Phi$ counts the number of elements of $W$ with fixed
circular descent number.  We also introduce a \emph{circular major
map} $\cmaj: W \rightarrow C$ which interacts in an interesting way
with $\cdes$.  In particular $\set{w \in W \mid \cmaj(w) = \id}$
gives a set of coset representatives for $W/C$.  In type $A_{n-1}$,
the circular major map generalizes the major index of $S_n$, taken
modulo $n$.

Weyl's formula for the order of the Weyl group $W$ states that
\[
|W| = f \cdot r! \cdot a_1a_2\cdots a_r
\]
where $f$ is the index of connection, $r$ is the rank of $\Phi$ and
$a_i$ are the coefficients of the simple roots in the maximal root
of $\Phi$.  We prove, using the geometry of alcoved polytopes, that
\[
\sum_{w\in W} q^{\cdes(w)}\,e^{\cmaj(w)} = \left(\sum_{x \in C}
e^{x}\right)\cdot A_r(q)\cdot [a_1]_q\cdots [a_r]_q\] where $A_r(q)$
is the usual Eulerian polynomial and $[n]_q$ denotes usual
$q$-analogues.  Here $e^x \in \Z[C]$ lies in the group algebra of $C$.

Finally, in analogy with Sturmfels' triangulation of the
hypersimplex in type $A$~\cite{Stu}, we study the toric ideals $I_P$
associated with an alcoved polytope $P$.  When $\Phi$ is one of the
root systems $A_n, C_n, D_4$, we give a Gr\"{o}bner basis $\g_P$ for
$I_P$ which induces the triangulation of $P$ into alcoves.


\section{Root System Notation}
We recall standard terminology related to root systems,
see~\cite{Hum} for more details. Let $V$ be a real Euclidean space
of rank $r$ with nondegenerate symmetric inner product
$(\lambda,\mu)$.  Let $\Phi\subset V$ be an irreducible {\it
crystallographic root system\/} with choice of basis of {\it simple
roots\/} $\alpha_1,\dots,\alpha_r$. Let $\Phi^+\subset \Phi$ be the
corresponding set of {\it positive roots\/} and $\Phi^-=-\Phi^+$ be
the set of {\it negative roots}. Then $\Phi$ is the disjoint union
of $\Phi^+$ and $\Phi^-$. We will write $\alpha>0$, for
$\alpha\in\Phi^+$; and $\alpha<0$, for $\alpha\in\Phi^-$. The
collection of {\it coroots\/}
$\alpha^\vee=2\alpha/(\alpha,\alpha)\in V$, for $\alpha\in\Phi$,
forms the {\it dual root system\/} $\Phi^\vee$. The {\it Weyl
group\/} $W\subset\mathrm{Aut}(V)$ is generated by the reflections
$s_\alpha:\lambda\mapsto \lambda-(\lambda,\alpha^\vee)\,\alpha$ with
respect to roots $\alpha\in\Phi$. The Weyl group $W$ is actually
generated by simple reflections $s_i = s_{\alpha_i}$ subject to the
Coxeter relations. The {\it length\/} $\ell(w)$ of an element $w\in
W$ is the length of a shortest decomposition for $w$ in terms of
simple reflections. There is a unique element $w_\circ\in W$ of
maximal possible length.

The {\it root lattice\/} $\root=\root(\Phi)\subset V$ is the integer
lattice spanned by the roots $\alpha\in\Phi$. It is generated by the
simple roots $\alpha_i$. The {\it weight lattice\/} is defined by
$\weight = \weight(\Phi)=\{\lambda\in V\mid (\lambda,\alpha^\vee)\in
\Z, \textrm{ for all } \alpha\in \Phi\}$. The weight lattice
$\weight$ contains the root lattice $\root$ as a subgroup of finite
index $f$. The quotient group $\weight/\root$ is isomorphic to the
center of the universal simply-connected Lie group $G^\vee$
associated with the root system $\Phi^\vee$. The index
$f=|\weight/\root|$ is called the {\it index of connection}.

The {\it coroot lattice} is the integer lattice $\coroot =
\coroot(\Phi)$ spanned by the coroots $\alpha_i^{\vee}$.  The {\it
coweight lattice} $\coweight$ is the integer lattice defined by
$\coweight = \coweight(\Phi) = \set{\lambda \in V \mid (\lambda,
\alpha) \in \Z, \textrm{ for all } \alpha \in \Phi}$.  Let
$\omega_1,\dots,\omega_r\subset V$ be the basis dual to the basis of
simple roots $\alpha_1,\dots,\alpha_r$, i.e.,
$(\omega_i,\alpha_j)=\delta_{ij}$. The $\omega_i$ are called the
{\it fundamental coweights}. They generate the coweight lattice
$\coweight$.

Let $\rho=\omega_1+\cdots+\omega_r$. The {\it height\/} of a root
$\alpha$ is the number $(\rho,\alpha)$ of simple roots that add up
to $\alpha$.  Since we assumed that $\Phi$ is irreducible, there
exists a unique {\it highest root\/} $\theta\in\Phi^+$ of maximal
possible height. For convenience we set $\alpha_0 = -\theta$. Let
$a_0 = 1$ and $a_1,\dots,a_r$ be the positive integers given by
$a_i=(\omega_i,\theta)$, or, equivalently, $a_0\alpha_0 +
a_1\alpha_1+\cdots+a_r\alpha_r = 0$. The {\it dual Coxeter number\/}
is defined as $h^\vee=(\rho,\theta)+1= a_0 + a_1+\cdots+a_r$.

\begin{lemma}
\label{lem:Weylaction} Let $\ll \in \coweight$ be an integral
coweight and $w \in W$ be a Weyl group element.  Then
\[
w(\ll) - \ll \in \coroot.
\]
\end{lemma}

\begin{proof}
Since $\coroot \subset \coweight$ it suffices to check this for a
simple reflection $s_\alpha$.  We compute $s_\alpha(\ll) - \ll =
(\ll,\alpha)\,\alpha^\vee \in \coroot$.
\end{proof}

\section{The Affine Weyl group and Alcoved Polytopes}
\label{sec:alcoved} The {\it affine Weyl group\/} $\Waff$ associated
with the root system $\Phi$ is generated by the reflections
$s_{\alpha,k}:V\to V$, $\alpha\in\Phi$, $k\in\Z$, with respect to
the affine hyperplanes
$$
H_{\alpha,k}=\{\lambda\in V\mid (\lambda,\alpha)=k\}.
$$

The coweight lattice $\coweight$ and coroot lattice $\coroot$ act on
the space $V$ by translations.  We will identify $\coweight$ and
$\coroot$ with these groups of translations.
The Weyl group $W$ normalizes these groups.

\begin{lemma} {\rm \cite{Hum}} \
The affine Weyl group $\Waff$ is the semidirect product $W\ltimes
\coroot$  of the usual Weyl group $W$ and the coroot lattice
$\coroot$. \label{lem:semidirect-product}
\end{lemma}

The connected components of the complement to these hyperplanes
$V\setminus \bigcup H_{\alpha,k}$ are called {\it alcoves}. Let $\A$
be the set of all alcoves. A {\it closed alcove\/} is the closure of
an alcove. Each alcove $A$ has the following form:
$$
A=\{\lambda\in V\mid m_\alpha<(\lambda,\alpha)<m_\alpha+1, \textrm{
for }\alpha\in\Phi^+\},
$$
where $m_\alpha=m_\alpha(A)$ is a collection of integers associated
with the alcove $A$.

\begin{lemma} {\rm \cite{Hum}} \
The affine Weyl group $\Waff$ acts simply transitively on the
collection $\A$ of all alcoves. \label{lem:simply-transitive}
\end{lemma}

The fundamental alcove is the simplex given by
\begin{equation}
\begin{array}{l}
A_\circ= \{\lambda\in V\mid 0<(\lambda,\alpha)<1,\,
\textrm{for }\alpha\in\Phi^+\}\\[.1in]
\quad\ = \{\lambda\in V \mid(\lambda,\alpha_i)>0,\textrm{ for
}i=1,\dots,r;
\textrm{ and }(\lambda,\theta)<0\}\\[.1in]
\quad\ = \{x_1\omega_1+\cdots + x_r\omega_r\mid x_1,\dots,x_r>0
\textrm{ and }
a_1 x_1+\cdots +a_r x_r<1\}\\[.1in]
\quad\ = \textrm{Convex Hull of the points } 0,\ \omega_1/a_1,\
\dots\ ,\ \omega_r/a_r.
\end{array}
\label{eq:A_0}
\end{equation}
Lemma~\ref{lem:simply-transitive} implies that all alcoves are
obtained from $A_\circ$ by the action of $\Waff$. In particular, all
closed alcoves are simplices with the same volume. The closure of
the fundamental alcove $A_\circ$ is a fundamental domain of $\Waff$.
Let $\F\supset\A$ be the set of all faces of alcoves of all
dimensions.  We will think of elements of $\F$ as relatively open
sets, so that the space $V$ is the disjoint union of elements of
$\F$.

Our main object is defined as follows.

\begin{definition} An alcoved polytope $P$ is a convex polytope in
the space $V$ such that $P$ is a union of finitely many elements of
$\F$. A proper alcoved polytope is an alcoved polytope of top
dimension.
\end{definition}

By the definition, each proper alcoved polytope comes equipped with
a triangulation into closed alcoves.
The following Lemma is immediate from the definitions.

\begin{lemma} A bounded subset $P\subset V$ is an
alcoved polytope if and only if $P$ is the intersection of several
half-spaces of the form $\{\lambda\in V\mid (\lambda,\alpha)\geq
k\}$, for $\alpha\in\Phi$ and $k\in\Z$.
\end{lemma}

Let $(W,S)$ be a Coxeter group and $u, v \in W$.  A path from $u$ to
$v$ is a sequence $u = w_0 \to w_1 \to w_2 \to \cdots \to w_s = v$
such that $w_{i+1} = w_is$ for some simple reflection $s \in S$.  A
subset $K \subset W$ is called \emph{convex} if for every $u, v \in
K$ we have that any shortest path from $u$ to $v$ lies in $K$.

\begin{prop}
Let $P \subset V$ be a bounded subset which is a union of closed
alcoves.  Then $P$ is a convex polytope if and only if one of the
following conditions hold:
\begin{enumerate}
\item
\label{it:convex1} For any two alcoves $A, B \subset P$, any
shortest path $A = A_0 \to A_1 \to A_2 \to \cdots \to A_s = B$ lies
in $P$.  Here $A_i \in \A$ are alcoves and $A' \to A''$ means that
the closures of the two alcoves $A'$ and $A''$ share a facet.
\item
\label{it:convex2} The subset $W_P = \set{w \in \Waff \mid
w(A_\circ) \subset P}$ of the affine Weyl group is a convex subset.
\end{enumerate}
\end{prop}
\begin{proof}
Suppose $P$ is a polytope and $A = A_0 \to A_1 \to A_2 \to \cdots
\to A_s = B$ is a shortest path with $A, B \in P$.  Suppose that
$A_{i}$ lies in $P$ but $A_{i+1}$ lies outside of $P$.  Then the
hyperplane $H_{\alpha,k}$ which separates $A_{i}$ and $A_{i+1}$ must
be a facet of $P$ so eventually the shortest path must cross
$H_{\alpha,k}$ again, say $A_{j}$ lies on the same side of
$H_{\alpha,k}$ as $A_{i+1}$ and $A_{j+1}$ lies on the other side.
Then reflecting the path $A_{i+1} \to A_{i+2} \to \cdots \to A_{j}$
in the hyperplane $H_{\alpha,k}$, we get another path from $A$ to
$B$ which is shorter.  Conversely, suppose condition
(\ref{it:convex1}) holds but $P$ is not convex.  This implies that
there are alcoves $A, B \in P$ and points $a \in A$, $b \in B$ such
that the straight line $\overline{ab}$ does not lie in $P$. Since
$P$ is compact we may assume $a$ and $b$ lie in the interior of $A$
and $B$ respectively.  Thus without loss of generality we may assume
that the line $\overline{ab}$ does not intersect any face of $\F$ of
codimension more than one.  The sequence of alcoves obtained by
travelling along $\overline{ab}$ must be a shortest path.  This is
clear as every hyperplane $H_{\alpha,k}$ that $\overline{ab}$
intersects separates $A$ from $B$ and so must be a separating
hyperplane between some two alcoves $A_i$ and $A_{i+1}$ in any path
from $A$ to $B$.

Condition~\ref{it:convex2} follows immediately from translating the
action of the simple generators of $\Waff$ on alcoves.
\end{proof}

Thus each alcoved polytope is of the form
$$
P = \{\lambda\in V \mid k_\alpha\leq (\lambda,\alpha) \leq
K_\alpha,\ \textrm{for }\alpha\in\Phi^+\},
$$
where $k_\alpha=k_\alpha(P)$ and $K_\alpha=K_\alpha(P)$ are two
collections of integers indexed by positive roots $\alpha\in\Phi^+$.

Let $\coweight/h^\vee= \{\lambda/h^\vee \mid \lambda\in\coweight\}$
be the weight lattice shrunk $h^\vee$ times.

\begin{lemma} {\rm\cite{Kos,LP}} \
For every alcove $A\in\A$, there is exactly one point of the lattice
$\coweight/h^\vee$ inside of $A$. For the fundamental alcove
$A_\circ$, we have $A_\circ \cap (\coweight/h^\vee) =
\{\rho/h^\vee\}$. \label{lem:central-point}
\end{lemma}

\begin{proof}  Since the affine Weyl group acts simply transitively
on $\A$ and preserves the lattice $\coweight/h^\vee$, it is enough
to prove the claim for the fundamental alcove $A_\circ$. In the
basis of fundamental coweights $\omega_1,\dots,\omega_r$, the
fundamental alcove $A_\circ$ is given by the inequalities
$x_1,\dots,x_r>0$ and $a_1 x_1+\cdots +a_r x_r<1$; and the lattice
$\coweight/h^\vee$ is given by $x_1,\dots,x_r\in\Z/h^\vee$. Recall
that $h^\vee = a_0 + a_1+\cdots+a_r$. Then the intersection
$A_\circ\cap (\coweight/h^\vee)$ consists of a single point with
coordinates $(1/h^\vee,\dots,1/h^\vee)$. In other words, this
intersection point is $(\omega_1+\cdots+\omega_r)/h^\vee =
\rho/h^\vee$.
\end{proof}

For $A\in\A$, we call the single element of $(\coweight/h^\vee)\cap A$
the {\it central point\/} of the alcove $A$. Let $Z =
(\coweight/h^\vee) \setminus \bigcup H_{\alpha,k}$ be the set of
central points of all alcoves, equivalently,
$$
Z=\{\lambda\in V\mid h^\vee\cdot (\alpha,\lambda)\equiv
1,\dots,h^\vee-1 \ (\textrm{mod}\, h^\vee), \textrm{ for
}\alpha\in\Phi^+\}.
$$
The set $Z$ of central points is in one-to-one correspondence with
the set $\A$ of alcoves.

\section{Weyl's Formula for the order of the Weyl group}
Let $\Vol$ be the volume form on the space $V$ normalized by
$\Vol(A_\circ)=1$.  Then the volume of any alcove is 1 and the
volume $\Vol(P)$ of an alcoved polytope $P$ is the number of alcoves
in $P$.  Equivalently, $\Vol(P)=|P\cap Z|$.

Let $\Pi$ be the alcoved polytope given by
$$
\begin{array}{l}
\Pi=\{\lambda\in V\mid 0\leq (\lambda,\alpha_i)\leq 1,
\textrm{ for }i=1,\dots,r\}\\[.1in]
\quad\ =\{x_1\omega_1+\cdots+x_r\omega_r \mid 0\leq x_i\leq
1,\textrm{ for } i=1,\dots,r\},
\end{array}
$$
i.e., $\Pi$ is the parallelepiped generated by the fundamental
coweights $\omega_1,\dots,\omega_r$. This polytope is a fundamental
domain of the coweight lattice $\coweight$. Since $A_\circ$ is the
simplex with the vertices $0,\omega_1/a_1,\dots, \omega_r/a_r$, we
have
$$
\Vol(\Pi)=\Vol(\Pi)/\Vol(A_\circ)={r!\,a_1\cdots a_r}.
$$
Thus the parallelepiped $\Pi$ consists of $r!\,a_1\cdots a_r$
alcoves.

Let $H$ be the alcoved polytope given by
$$
H=\{\lambda\in V\mid -1\leq (\lambda,\alpha)\leq 1, \textrm{ for
}\alpha\in\Phi^+\}.
$$
The polytope $H$ consists of all alcoves adjacent to the origin $0$,
i.e., it consists of the $|W|$ alcoves of the form $w(A_\circ)$, for
$w\in W$. In particular, its volume is the order of the Weyl group:
$\Vol(H)=|W|$. Lemma~\ref{lem:semidirect-product} implies that the
polytope $H$ is a fundamental domain of the coroot lattice
$\coroot$. Thus $\Vol(H)/\Vol(\Pi)=|\coweight/\coroot|=f$ is the
index of connection.   This implies the well-known formula for the
order of the Weyl group, see~\cite[4.9]{Hum}:
\begin{equation}
\label{eq:Weylsformula} |W| = f\cdot r!\cdot a_1\cdots a_r.
\end{equation}


\section{The group $C$}
\label{sec:groupC} For an integral coweight $\lambda\in\coweight$,
the affine translation $A+\lambda$ of an alcove $A$ is an alcove;
and the affine translation $P+\lambda$ of an alcoved polytope $P$ is
an alcoved polytope.

Let us define the equivalence relation ``$\sim$''
on the affine Weyl group $\Waff$ by
$$
u\sim w\textrm{ if and only if } u(A_\circ)=w(A_\circ)+\lambda,
\textrm{ for some } \lambda\in\coweight,
$$
where $u,w\in \Waff$. The relation ``$\sim$'' is invariant with
respect to the left action of the affine Weyl group. According to
Lemma~\ref{lem:central-point}, this equivalence relation can be
defined in terms of central points of alcoves as
$$
u\sim w\textrm{ if and only if }
u(\rho/h^\vee)-w(\rho/h^\vee)\in\coweight.
$$
Let $C$ be the subset of the usual Weyl group $W$ given by
$$
C=\{w\in W\mid w\sim 1\} = \{w\in W\mid w(\rho)-\rho\in
h^\vee\coweight\}.
$$
Also let $\Caff=\{w\in \Waff\mid w\sim 1\}$. Actually, $C$ is a
subgroup in $W$ and $\Caff$ is a subgroup in $\Waff$. Indeed, $u\sim
1$ and $w\sim 1$ imply that $uw\sim u\sim 1$.

The coroot lattice $\coroot$ is a normal subgroup in $\Caff$.
The group $\Caff$ is the semidirect product $C\ltimes \coroot$ and,
thus, $C\simeq \Caff/\coroot$.

Equivalence classes of elements of the Weyl group (respectively, the
affine Weyl group) with respect to the relation ``$\sim$'' are
exactly cosets in $W/C$ (respectively, $\Waff/\Caff$). Since $\Pi$
is a fundamental domain of the coweight lattice $\coweight$ and, for
an alcove $A\in\A$, there is a translation $A+\lambda$ such that
$A+\lambda=w(A_\circ)$, for some $w\in W$, we deduce that there are
natural one-to-one correspondences between the followings sets:
$$
W/C\simeq \Waff/\Caff\simeq \A/\coweight \simeq \{\textrm{alcoves in
}\Pi\}.
$$
In particular, the number of cosets $|W/C|$ equals $\Vol(\Pi) =
|W|/f$ and, thus, the order of the group $C$ is $|C|=f$.

There is a natural bijection $b:\coweight\to \Caff$ given by
$b(\lambda)=w$ whenever $w(A_\circ)=A_\circ+\lambda$.
Notice that $b$ may not be a homomorphism of groups.  However the
map $\bar b:\coweight \to C$ given by the composition of $b$ with
the natural projection $\Caff\to \Caff/\coroot\simeq C$ is a
homomorphism. Indeed, let $\bar b(\lambda)  = u$ and $\bar b(\mu) =
w$. Then $u,w\in W$ are given by $u(A_\circ) \equiv A_\circ +
\lambda \,\mod \coroot$ and $w(A_\circ)\equiv A_\circ+\mu \,\mod
\coroot$. Then $uw(A_\circ) \equiv A_\circ + \lambda + u(\mu) \equiv
A_\circ+\lambda+\mu \,\mod \coroot$. The last equation follows from
Lemma~\ref{lem:Weylaction}. The kernel of the map $\bar b$ is
$\coroot$. Thus $\bar b$ induces the natural isomorphism of groups:
$$
\coweight/\coroot \simeq C.
$$
The group $C$ is the cyclic group $\Z/(n+1)\Z$ in type $A_n$, a group of order $2$ for types $B_n$, $C_n$ and $E_7$, a group of order $4$ for type $D_n$, a group of order $3$ for type $E_6$, and trivial for $E_8$, $F_4$, and $G_2$.

\section{The statistic $\cdes$}

Let us say that a root $\alpha\in\Phi$ is an {\it inversion\/} of
Weyl group element $w\in W$ if $w(\alpha)<0$.  Equivalently, a
positive root $\alpha$ in an inversion of $w$ if and only if
$\ell(ws_\alpha)<\ell(w)$. Let us define
$$
\inv_\alpha(w) = \left\{
\begin{array}{cl}
0 & \textrm{if } w(\alpha)>0,\\
1 &  \textrm{if } w(\alpha)<0.
\end{array}\right.
$$

\begin{lemma}
\label{lem:inversionalcove} We have
$\inv_\alpha(w)=-m_\alpha(w^{-1}(A_\circ))$.
\end{lemma}
\begin{proof}
Indeed, $\alpha$ is an inversion of $w$ if and only if
$(w(\alpha),\rho)= (\alpha,w^{-1}(\rho))<0$, that is,
$-1<(\lambda,\alpha)<0$, for any $\lambda\in w^{-1}(A_\circ)$.
\end{proof}

Let $d_i(w) = \inv_{\alpha_i}(w)$, for $i=0,\dots,r$.  If $d_i(w) =
1$ we say that $w$ has a descent at $i$.

\begin{definition}
Let $w \in W$.  The \emph{circular descent number} $\cdes(w)$ is
defined by
$$
\cdes(w) = \sum_{i=0}^r a_i\,d_i(w).
$$
\end{definition}
Note that $\cdes(w)$ is always positive. Indeed if $d_i(w) = 0$ for
$i \in [1,r]$ then we have $w = 1$ and $d_0(w) = 1$.

Define $\delta_w \in \coroot$ for $w \in W$ by
\[
\delta_w = \sum_{i=0}^r d_i(w)\cdot\w_i
\]
where for convenience we let $\w_0 = 0$.

\begin{lemma}
\label{lem:delta} The coweight $\delta_w$ is the unique integral
coweight such that $w^{-1}(A_\circ) + \delta_w \in \Pi$.
\end{lemma}
\begin{proof}
Let $i \in [1,r]$.  Then by Lemma~\ref{lem:inversionalcove},
$-d_i(w)<(w^{-1}(\rho/h^\vee),\alpha_i)< 1-d_i(w)$ so that $\ll =
w^{-1}(A_\circ) + \delta_w $ satisfies $0 \leq (\ll,\alpha_i) \leq
1$.  The coweight $\delta_w$ must be unique since adding or
subtracting any fundamental coweight $\w_i$ will cause $\ll$ to
violate the inequality $0< (\ll,\alpha_i) < 1$.
\end{proof}

We set $\AA = \set{\alpha_0,\ldots,\alpha_r}$ and $\AA_i = \set{\alpha_j
\in \AA \mid a_j = i}$. For convenience we set $J = \set{j \in [0,r]
\mid a_j = 1}$.

\begin{prop}
\label{prop:differentC} We have the following equivalent
descriptions of the group $C \subset W$.
\begin{align}
\label{eq:c1}
C &= \set{w \in W \mid \cdes(w) = 1} \\
\label{eq:c2}
  &= \set{w \in W \mid w(\AA) = \AA} \\
\label{eq:c3}
  &= \set{w \in W \mid w(\AA_k) = \AA_k \mbox{ for all $k$ }}.
\end{align}
For any $j \in J$ there exist a unique Weyl group element
$w^{(j)}\in C$ such that $w(\alpha_i)>0$, for $i\ne j$ and
$w(\alpha_j) < 0$.
\end{prop}

\begin{proof}
Let $c \in C$.  By Lemma~\ref{lem:delta} and the definition of $C$,
we see that $c^{-1}(A_\circ) + \delta_c = A_\circ$.  By
Lemma~\ref{lem:central-point}, this implies that
$c^{-1}(\rho/h^\vee) + \delta_c = \rho/h^\vee$.  It is clear that
$(\delta_c, \theta) = \cdes(c) - d_0(c)$.  We compute, using
Lemma~\ref{lem:inversionalcove}, that
\begin{align*}
d_0(c) - 1 \leq (c^{-1}(\rho/h^\vee),\theta) \leq d_0(c).
\end{align*}
So summing we have
\begin{align*}
\cdes(c) - 1 \leq (\rho/h^\vee, \theta ) \leq \cdes(c)
\end{align*}
which immediately implies that $\cdes(c) = 1$ since
$(\rho/h^\vee,\theta) = 1-1/h^\vee$.  The converse follows in the
same manner.  This establishes the equality in (\ref{eq:c1}).

Now suppose $c \in C$.  We establish (\ref{eq:c2}).  By definition, $c(\alpha_j) < 0$ for some $j \in J$ and $c(\alpha_i) > 0$ for $i
\neq j$.  Let $\AA_{\neq j} = \set{\alpha_i \in \AA | i \neq j}$.  We have $c^{-1}(A_\circ)+\omega_j = A_\circ$, and in particular the set $\{0,\omega_1/a_1,\ldots,\omega_r/a_r\}$ is sent to itself under the map $\lambda \mapsto c^{-1}(\lambda) + \omega_j$.  Substituting this fact into $(c(\alpha_i),\omega_k/a_k) = (\alpha_i,c^{-1}(\omega_k/a_k))$ and noting that $\{\omega_1,\ldots,\omega_n\}$ are a dual basis to $\{\alpha_1,\ldots,\alpha_n\}$, we deduce that $c(\AA_{\neq j}) = \AA_{\neq 0}$, and $c(\alpha_j) = \alpha_0$.  Thus (\ref{eq:c2}) holds.

We get (\ref{eq:c3}) from (\ref{eq:c2}) by
noting that up to scalar multiplication the relation $\sum_i a_i
\alpha_i = 0$ is the only linear dependence amongst the roots in $\AA$.


Conversely, (\ref{eq:c3}) clearly implies (\ref{eq:c1}) by the
definition of $\cdes$.  The last statement of the proposition also
follows from this discussion.
\end{proof}

By property (\ref{eq:c1}), we have  $f = |C| = \#\{i \in [0,r] \mid
a_i = 1\}$.  These $i$'s correspond to minuscule coweights
$\w_i$.  Recall that a minuscule weight is one whose weight
polytope has no internal weights.

We remark that the group $C$ was previously studied by
Verma~\cite{Ver} but not in the current context of the statistic
$\cdes$. The group $C$ is related to the statistic $\cdes$ on the
whole of $W$ in an intimate way.

\begin{theorem}\label{thm:cdes}
The statistic $\cdes$ is constant on the double cosets $C\backslash
W/C$.
\end{theorem}
\begin{proof}
Let $w \in W$ and $c \in C$ so that $c(\alpha_j) = \alpha_0$ for $j
\in J$.  We need to prove that $\cdes(cw) = \cdes(w) = \cdes(wc)$.
The latter equality is immediate from condition (\ref{eq:c3}) of
Proposition~\ref{prop:differentC}.

Let $\alpha \in \AA$ and let $\beta = w(\alpha) = b_1 \alpha_1 +
\cdots + b_r \alpha_r$. The $b_i$ are either all positive or all
negative.  The element $w$ has a descent at $\alpha$ if and only if
$\beta < 0$.  Now
\begin{equation}
\label{eq:cw} cw(\alpha) = c(\beta) = b_1 c(\alpha_1) + \cdots + b_r
c(\alpha_r).
\end{equation}
If $b_j = 0$ then clearly $d_\alpha(w) = d_\alpha(cw)$.  If $b_j
\neq 0$, then we have a term of the form
\begin{equation}
\label{eq:bj} b_j c(\alpha_j) = b_j \alpha_0 = - b_j (a_1 \alpha_1 +
\cdots + a_r \alpha_r)
\end{equation}
in equation (\ref{eq:cw}).  Since $-\alpha_0$ is the longest root we
have $|a_k| \geq |b_k|$. Thus substituting (\ref{eq:bj}) into
(\ref{eq:cw}) we see that $b_j \neq 0$ implies that $d_\alpha(w) = 1
- d_\alpha(cw)$.  Indeed $b_j \in \set{0,1,-1}$ and we have
$d_\alpha(w) - d_\alpha(cw) = b_j$.

Now we have
\[
0 = a_0 w(\alpha_0) + \cdots + a_r w(\alpha_r)
\]
and so expressing both sides in terms of the simple roots $\alpha_1,
\ldots, \alpha_r$ we see that the coefficient of $\alpha_j$ is 0.
Write $b_j(\alpha)$ for the coefficient $b_j$ in the proof earlier,
obtaining the equality
\[ 0 = \sum_{i=0}^r a_i b_j(\alpha_i) = \sum_{i = 0}^r \left(d_i(cw) - d_i(w)\right) = \cdes(cw) -\cdes(w).\]
\end{proof}

\section{The map $\cmaj$}

Define the {\it circular major map} $\cmaj: W \rightarrow C$ by
\[
\cmaj(w) = \bar b(\delta_w)
\]
where $\bar b$ is the isomorphism from Section~\ref{sec:groupC}.

\begin{lemma}
The map $\cmaj$ satisfies $\cmaj(c) = c$ for $c \in C$.
\end{lemma}
\begin{proof}
By Lemma~\ref{lem:delta}, we have $c(A_\circ) = A_\circ +
c(\delta_c)$. But by Lemma~\ref{lem:Weylaction}, $c(\delta_c) =
\delta_c \mod \coroot$ so that by definition $\bar b(\delta_c) = c$.
\end{proof}

\begin{theorem}
\label{thm:cmajcdes} The map $\cmaj$ satisfies
\[
\cmaj(c_1 w c_2) = c_1 w c_2^{\cdes(w)}
\]
for $c_i \in C$ and $w \in W$.
\end{theorem}
\begin{proof}
Let $w \in W$ and $c = w^{(j)} \in C$, so that $c^{-1}(\alpha_0) =
\alpha_j$.  Thus $c(\alpha_j) < 0$ so that $\cmaj(c) = \w_j$.  We
first consider $wc$. We have
\[
w^{-1}(\rho/h^\vee) + \delta_w = \mu/h^\vee
\]
for some $\mu \in \Lambda$ satisfying $0 < (\mu, \alpha_i) <
h^\vee$ for $i \in [1,r]$. Applying $c^{-1}$ on the left to both
sides we obtain
\[
c^{-1}w^{-1}(\rho/h^\vee) + c^{-1}(\delta_w) = c^{-1}(\mu/h^\vee).
\]
Now $0 < (c^{-1}(\mu/h^\vee), c^{-1}(\alpha_i)) < 1$ and as $i$
varies, we obtain every simple root in the form $c^{-1}(\alpha_i)$
apart from $c(\alpha_0) = \alpha_j$ for some $j \in J$ say.  But $-
\cdes(w) < (\mu/h^\vee, \alpha_0) < 1 - \cdes(w)$, so we have $0 <
(c^{-1}(\mu/h^\vee)+ \cdes(w)\,.\,\w_j, \alpha_j) < 1$.

Thus
\begin{align*}
\delta_{wc} &= c(\delta_w) + \cdes(w)\,.\,\w_j \\
\cmaj(wc) &\equiv \cmaj(w) + \cdes(w)\,.\,\w_j \ \mod \coroot.
\end{align*}
Hence $\cmaj(wc) = \cmaj(w)c^{\cdes(w)}$.  We have used
Lemma~\ref{lem:Weylaction}.

Similarly,
\[
c^{-1}(\rho/h^\vee) + \delta_c = \rho/h^\vee
\]
gives
\begin{align*}
w^{-1}c^{-1}(\rho/h^\vee) + w^{-1}(\w_j) &= w^{-1}(\rho/h^\vee) \\
w^{-1}c^{-1}(\rho/h^\vee) + w^{-1}(\w_j) + \delta_w &=
w^{-1}(\rho/h^\vee) + \delta_w = \rho/h^\vee.
\end{align*}
 This implies that
\begin{align*}
\delta_{cw} &= w^{-1}(\w_j) + \delta_w \\
\cmaj(cw) &= \cmaj(w) + \w_j \ \mod \coroot.
\end{align*}
Hence $\cmaj(cw) = c\cdot \cmaj(w)$.

\end{proof}
Theorem~\ref{thm:cmajcdes} shows that the map $\cmaj$ allows us to
pick representatives for the right cosets $W/C$.  For example
$\set{w \mid \cmaj(w) = \id}$ is a set of right coset
representatives.  In type $A$, $\cmaj$ has an explicit
representation theoretic meaning, see
Theorem~\ref{thm:inducedcycliccharacter}.

\section{Relation between volumes and numbers of lattice points}

Let $P$ be an alcoved polytope, and let $A\in\A$ be an alcove. Let
$k_\alpha=k_\alpha(P)$, $K_\alpha=K_\alpha(P)$, and
$m_\alpha=m_\alpha(A)$, for $\alpha\in\Phi^+$ be as in
Section~\ref{sec:alcoved}. Let us define the alcoved polytope
$P_{(A)}$ as
$$
P_{(A)} = \{\lambda\in V\mid
k_\alpha-m_\alpha\leq(\lambda,\alpha)\leq
K_\alpha-m_\alpha-1,\textrm{ for all }\alpha\in\Phi^+\}.
$$

The following claim follows directly from the definitions.

\begin{lemma}  For $P$ and $A$ as above, the set
$P_{(A)}\cap \coweight$ of lattice points in $P_{(A)}$ is exactly
the set of integral coweights $\lambda\in\coweight$ such that
$A+\lambda$ is an alcove in $P$. \label{lem:A+lambda}
\end{lemma}


The lemma says that lattice points in $P_{(A)}$ are in one-to-one
correspondence with alcoves in $P$ that are obtained by affine
translations of $A$.

For $w\in W$, the definition of the polytope
$P_{(w)}=P_{(w(A_\circ))}$ can be rewritten as
$$
P_{(w)} = \{\lambda\in V\mid
k_\alpha+d_\alpha(w^{-1})\leq(\lambda,\alpha)\leq
K_\alpha+d_\alpha(w^{-1})-1,\textrm{ for all }\alpha\in\Phi^+\}.
$$
We have used Lemma~\ref{lem:inversionalcove}.

Notice that $P_{(A+\lambda)} = P_{(A)}-\lambda$. Thus the polytopes
$P_{(A)}\equiv P_{(B)}$ are equivalent modulo affine translations by
elements of $\coweight$, whenever $A\equiv B\ \mathrm{mod}\
\coweight$. This implies that the polytope $P_{(\bar w)}=P_{(w)}$ is
correctly defined modulo affine translations by coweights
$\lambda\in \Lambda$, where $w\in W$ is any representative of a
coset $\bar w\in W/C$.

Let $I(P)=|P\cap \Lambda|$ be the number of lattice points in $P$.
The following statement establishes a relation between the volume of
an alcoved polytope and the numbers of lattice points in smaller
alcoved polytopes.

\begin{theorem}  Let $P$ be an alcoved polytope.
Then
$$
\Vol(P) = \sum_{\bar w\in W/C} I(P_{(\bar w)}).
$$
\label{th:vol=I}
\end{theorem}

\begin{proof}
According to Lemma~\ref{lem:A+lambda}, the total number of alcoves
in $P$ equals the sum of $I(P_{(A)})$ over representatives $A$ of
cosets $\A/\coroot$. This is exactly the claim of the theorem.
\end{proof}

\section{Generalized hypersimplices}

For $k=1,\dots,h^\vee-1$, let us define the $k$-th {\it generalized
hypersimplex\/} $\Delta_k^\Phi$ as the alcove polytope given by
$$
\Delta_k^\Phi=\{\lambda\in V\mid 0\leq (\lambda,\alpha_i)\leq 1,
\textrm{ for }i=1,\dots,r; \textrm{ and } k-1\leq
(\lambda,\theta)\leq k\}.
$$
In other words, the generalized hypersimplices are the slices of the
parallelepiped $\Pi$ by the parallel hyperplanes of the form
$H_{\theta,k}$, for $k\in\Z$. Clearly, the first generalized
hypersimplex is the fundamental alcove: $\Delta_1^\Phi=A_\circ$.
Also the last generalized hypersimplex is the alcove given by
$\Delta_{h^\vee-1}^\Phi=w_\circ(A_\circ)+\rho$, where $w_\circ\in W$
is the longest element in $W$.

\begin{lemma}  Let $w\in W$.
The polytope $(\Delta_k^\Phi)_{(w)}$ consists of a single point
$\lambda\in\coweight$, if $\cdes(w^{-1})=k$; and
$(\Delta_k^\Phi)_{(w)}$ is empty, if $\cdes(w^{-1})\ne k$.
\end{lemma}

\begin{proof} By the definition, the polytope $(\Delta_k^\Phi)_{(w)}$
is given by
$$
(\Delta_k^\Phi)_{(w)}=\{\lambda\in \Lambda\mid d_{i}(w^{-1})=
(\lambda,\alpha_i), \textrm{ for }i=1,\dots,r; \
k-d_0(w^{-1})=(\lambda,\theta)\}.
$$
The first $r$ equations $d_i(w^{-1})=(\lambda,\alpha_i)$ have a
single solution $\lambda = \sum_{i\in D}\omega_i = \delta_{w^{-1}}$,
where $D=\{i\mid d_{\alpha_i}(w^{-1})=1\}$. The last equation
$k-d_0(w^{-1})=(\lambda,\theta^\vee)$, for the point
$\lambda=\sum_{i\in D} \omega_i$, says that $\cdes(w)=k$.
\end{proof}

\begin{corollary}  All representatives $w$ of a coset $\bar w\in W/C$
have the same generalized descent numbers $\cdes(w^{-1})$.
\end{corollary}

\begin{proof} The polytopes $(\Delta_k^\Phi)_{(u)}\equiv (\Delta_k^\Phi)_{(w)}$
are equivalent modulo affine translations, whenever $\bar u = \bar
w$ in $W/C$.
\end{proof}

Define $\cdes(\bar w^{-1}) = \cdes(w^{-1})$, where $w\in W$ is any
representative of a coset $\bar w$. Theorem~\ref{th:vol=I} implies
the following statement.

\begin{theorem}\label{T:hypersimplex}  The volume $\Vol(\Delta_k^\Phi)$ of $k$-th generalized
hypersimplex $\Delta_k^\Phi$ equals the number of cosets $\bar w\in
W/C$ such that $\cdes(\bar w^{-1})= k$.  Equivalently, $f\cdot
\Vol(\Delta_k^\Phi)$ equals the number of elements $w\in W$ such
that $\cdes(w^{-1})=k$.
\end{theorem}

Let $\H(b_1,\dots,b_r;k,K)$ be the \emph{thick hypersimplex} given
by
$$
\{\lambda\in V\mid 0\leq (\lambda,\alpha_i)\leq b_i, \textrm{ for
}i=1,\dots,r; \textrm{ and } k\leq (\lambda,\theta)\leq K\}.
$$

\begin{proposition}
We have
$$\Vol(\H(b_1,\dots,b_r;k,K) =
\sum_l \Vol(\Delta_l^\Phi)\cdot I(\H(b_1-1,\dots,b_r-1;l-K+1,l-k)).
$$
\end{proposition}
\begin{proof}
Let $\ll \in \coweight$ be in the interior of $\Delta_l^\Phi$ and
$\mu \in \coweight$.  Then $0 < (\ll,\alpha_i) < 1$ for $i \in
[1,r]$ and $l-1 < (\ll,\theta) \leq l$. Thus $\ll+ \mu \in
\H(b_1,\dots,b_r;k,K)$ if and only if we have $0 \leq (\mu,\alpha_i)
\leq b_i -1$ and $l - K + 1 \leq (\mu,\theta) \leq l-k$.  We
conclude that for any alcove $A \in \Delta_l^\Phi$, we have
$\H(b_1,\dots,b_r;k,K)_{(A)}$ = $\H(b_1-1,\dots,b_r-1;l-K+1,l-k) \
\mod \ \coweight$.  As $l$ varies, we obtain a translate of
$\H(b_1,\dots,b_r;k,K)_{(\bar w)}$ for each coset $\bar{w}$ exactly
once in this form.
\end{proof}

\section{A $q$-analogue of Weyl's formula}
Recall that for a permutation $w=w_1\dots w_n$ in the symmetric
group $S_n$, a {\it descent\/} is an index $i\in\{1,\dots,n-1\}$
such that $w_i>w_{i+1}$.  Let $\des(w)$ be the number of descents of
$w\in S_n$. The $n$-th {\it Eulerian polynomial\/} $A_n(q)$ is
defined as
$$
A_n(q) = \sum_{w\in S_n} q^{\des(w)+1},
$$
for $n\geq 1$, and $A_0(q) = 1$.  These polynomials can also be
expressed as $A_n(q) = (1-q)^{n+1}\cdot \sum_{k\geq 0} k^n \,q^k$.
Let $[n]_q = (1-q^n)/(1-q)$ denote the $q$-analogue of an integer
$n\in\Z$.

The group algebra $\Z[\Lambda^\vee/L^\vee]$ has a $\Z$-basis of formal exponents
$e^{x}$, for $x\in \Lambda^\vee/L^\vee$, with multiplication $e^x\cdot e^y =
e^{x+y}$. Let $\Z[q][\Lambda^\vee/L^\vee]=\Z[q]\otimes\Z[\Lambda^\vee/L^\vee]$.  The
following theorem generalizes Weyl's formula (\ref{eq:Weylsformula})
for the order of the Weyl group.

\begin{theorem}\label{T:qWeyl}  The following identity holds in the group algebra
$\Z[q][\Lambda^\vee/L^\vee]$:
$$
\sum_{w\in W} q^{\cdes(w)}\,e^{\cmaj(w)} = \left(\sum_{x \in
\Lambda^\vee/L^\vee} e^{x}\right)\cdot A_{r}(q)\cdot [a_1]_q\cdots [a_r]_q.
$$
In particular, we have the following identity for polynomials in
$\Z[q]$:
$$
\sum_{w\in W} q^{\cdes(w)} = f\cdot A_r(q)\cdot [a_1]_q\cdots
[a_r]_q.
$$
\end{theorem}

We first establish the following generating function for the volumes of generalized hypersimplices.

\begin{prop}\label{P:Pi}
The generating function for the volumes of generalized hypersimplices is given by
$$
\sum_{k =1}^{h^\vee-1} \Vol(\Delta_k^\Phi) q^k = A_{r}(q)\cdot [a_1]_q\cdots [a_r]_q.
$$
\end{prop}
\begin{proof}
The union of the generalized hypersimplices $\Delta_k^\Phi$ for $k = 1,2,\ldots,h^\vee -1$ is the fundamental parallelpiped $\Pi$.  For a bounded polytope $P \subset V$, define the generating function
$$
g_P(q) = \sum_k \Vol(P \cap \{\lambda \in V \mid k-1 \leq (\lambda,\theta) \leq k\}) \, q^k \in \R[q^{\pm 1}].
$$
Then $g_\Pi(q) = \sum_{k =1}^{h^\vee-1} \Vol(\Delta_k^\Phi) q^k$.  We note that if $(\lambda,\theta)= a \in \Z$, then $g_{P+\lambda}(q) = q^a \, g_P(q)$.  Now set $\Xi$ to be the parallelpiped spanned by the vectors $\omega_1/a_1,\ldots,\omega_r/a_r$.  Then $\Pi$ is a union of translates of $\Xi$ by integral linear combinations of the vectors $\omega_i/a_i$, and we deduce that
$$
g_\Pi(q) = g_\Xi(q) \cdot [a_1]_q\cdots [a_r]_q.
$$
Since we are normalizing the fundamental alcove $A_\circ$ with vertices $\omega_1/a_1,\ldots,\omega_r/a_r$ to have $\Vol(A_\circ) = 1$, it follows that $g_\Xi(q)$ is equal to the generating function of the normalized volumes of the usual hypersimplices:
$$
g_\Xi(q) = \sum_{k=1}^{r} \Vol([0,1]^r \cap \{(x_1,\ldots,x_r) \in \R^r \mid k-1 \leq x_1 + \cdots + x_r \leq k\}) \, q^k
$$
which is well known to equal to the Eulerian polynomial $A_r(q)$.  This also follows from Theorem \ref{T:hypersimplex} (see Section \ref{sec:typeA}) and is studied in detail in \cite{API}.
\end{proof}

\begin{proof}[Proof of Theorem \ref{T:qWeyl}]
Using Theorem \ref{thm:cmajcdes}, we let $W' = \{w \in W \mid \cmaj(w) = \id\}$ be a set of left coset representatives for $C\backslash W$.  Then $(W')^{-1}$ is a set of right coset representatives for $W/C$.  We calculate
\begin{align*}
&\sum_{w\in W} q^{\cdes(w)}\,e^{\cmaj(w)} \\
& = \sum_{w \in W'} q^{\cdes(w)} \, \sum_{c \in C} e^{\cmaj(c)} &\mbox{by Theorems \ref{thm:cdes} and \ref{thm:cmajcdes} } \\
&= \sum_{w \in (W')^{-1}}\, q^{\cdes(w^{-1})} \cdot \left(\sum_{x \in
\Lambda^\vee/L^\vee} e^{x}\right) \\
& = \left(\sum_{k =1}^{h^\vee-1} \Vol(\Delta_k^\Phi) \, q^k\right)\cdot \left(\sum_{x \in
\Lambda^\vee/L^\vee} e^{x}\right)&\mbox{by Theorem \ref{T:hypersimplex} }\\
&=\left(\sum_{x \in
\Lambda^\vee/L^\vee} e^{x}\right)\cdot A_{r}(q)\cdot [a_1]_q\cdots [a_r]_q & \mbox{by Proposition \ref{P:Pi}}
\end{align*}
\end{proof}

\begin{remark}
We have
$$
\sum_{w\in W} q^{\cdes(w)}\,e^{\cmaj(w)} = \sum_{w\in W} q^{\cdes(w)}\,e^{\cmaj(w^{-1})}.
$$
This follows from the fact (Theorem \ref{thm:cdes}) that $\cdes$ is constant on $C \backslash W /C$ double cosets. Each double coset is a disjoint union of left (resp. right) cosets $C \backslash W$ (resp. $W/C$) for which $e^\cmaj(w)$ (resp. $e^\cmaj(w^{-1})$) takes the values $\left(\sum_{x \in
\Lambda^\vee/L^\vee} e^{x}\right)$.
\end{remark}

\begin{remark}
It would be interesting to compare Theorem \ref{T:qWeyl} with Stembridge and Waugh's Weyl group identity \cite{SW}.
\end{remark}

The following question seems interesting.
\begin{question}
What is $\sum_{w\in W} x^{\cmaj(w)} y^{\cmaj(w^{-1})} $ in
$\Z[\Lambda^\vee/L^\vee]\otimes \Z[\Lambda^\vee/L^\vee]$?
\end{question}

\section{Example: type $A$}
\label{sec:typeA} Let $\Phi = A_{n-1} \subset \R^n/\R(1,1,\cdots,1)$
throughout this section.  The simple roots are $\alpha_i = e_i -
e_{i+1}$ where $e_i$ are the coordinate basis vectors of $\R^n$. The
longest root is $\theta = e_1-e_n$ and we have $a_i = 1$ for $i \in
[0,n]$. The Weyl group $W = S_n$ is the symmetric group on $n$
letters and $\cdes(w)$ is equal to the usual number of descents of
$w$ plus a descent at $n$ if $w_n
> w_1$.  This is the reason for calling $\cdes$ the circular
descent number.  The group $C = \left<c = (123\cdots (n-1)n)\right>$
is generated by the long cycle. The fundamental coweights are given
by $\w_i = e_1 + e_2 + \cdots + e_i$ and one can check that
$\delta_{c^i} = \w_i$.  Thus $\cmaj(w) = c^{-\maj(w) \mod n}$ where
$\maj(w)$ denotes the usual major index of $w$.  We can verify
Proposition~\ref{thm:cmajcdes} directly: left multiplication by the
long cycle $c$ maps $w_1 w_2 \cdots w_n$ to
$(w_1+1)(w_2+1)\cdots(w_n+1)$ where $`n+1'$ is identified with
$`1'$.  Right multiplication by $c$ maps $w_1 w_2 \cdots w_n$ to
$w_2w_3\cdots w_n w_1$.  

The following theorem (\cite{KW} and
\cite[Ex. 7.88]{EC2}) suggests that the map $\cmaj$ may have an
explicit representation theoretic interpretation.  Let $\chi^\lambda$ denote the irreducible character of the symmetric group $S_n$ labeled by a partition $\lambda$.

\begin{theorem}
\label{thm:inducedcycliccharacter} Let $C_n \subset S_n$ be a cyclic
subgroup of order $n$.  Let $\rho = \mathrm{ind}_{C_n}^{S_n} e^{2\pi
\sqrt{-1} j/n}$ be an induced character of $S_n$.  Then we have
 \[ \left<\rho,\chi^\lambda\right>
= \#\{{\rm SYT}(T)\mid \sh(T)=\lambda \text{ and } \, \maj(T) \equiv j \
\mod n\}. \] Here a descent of a standard young tableau (SYT) $T$ is an index $i$ such that the box
containing $i+1$ is to the southwest of the box containing $i$ in
$T$.  The index $\maj(T)$ is defined to be the sum of all the
descents of $T$.
\end{theorem}

It is not hard to see that the polytopes $\Delta_k^{A_{n-1}}$ are
affinely equivalent to the usual hypersimplices defined as the
convex hull of the points $\epsilon_I$ where $\epsilon_I = \sum_{i
\in I} e_i$ and $I$ varies over all $k$-subsets of $[n]$.  The
alcoved triangulation here is identical to that studied
in~\cite{API}.  

\section{Example: Type $C$}
\label{sec:typeC} Let $\Phi = C_n$ with $2n$ long roots $\pm 2e_i$
for $1 \leq i \leq n$ and $2n(n-1)$ short roots $\pm e_i \pm e_j$
for $1 \leq i \neq j \leq n$. A system of simple roots is given by
$\alpha_1 = e_1 - e_2, \alpha_2 = e_2 - e_3, \ldots, \alpha_{n-1} =
e_{n-1}-e_n, \alpha_n = 2e_n$. Then $\theta = 2e_1 = 2\alpha_1 +
2\alpha_2 + \cdots + 2\alpha_{n-1} + \alpha_n$, so that $a_0 = a_n = 1$ and $a_i = 2$ for $1\leq i \leq n-1$.  The fundamental
coweights are given by $\omega_1 = e_1$, $\omega_2 = e_1 + e_2$,
$\ldots$, $\omega_{n-1} = e_1 + \cdots + e_{n-1}$, $\omega_n =
1/2(e_1+\cdots+e_n)$.

We identify the Weyl group $W$ of type $C_n$ with the group of
signed permutations $w_1w_2\cdots w_n$ in the usual way: $w_i \in
\pm\set{1,2,\ldots,n}$ and $|w_1||w_2|\cdots|w_n|$ is a usual
permutation in $S_n$.  For $i \in [1, n-1]$ a signed permutation $w
= w_1w_2\cdots w_n$ has a descent at $i$ if $w_i > w_{i+1}$, as
usual.  We have a descent at $0$ if $w_1 > 0$ and a descent at $n$
if $w_n < 0$.  The group $C$ has order two, with unique non-identity
element $c = (-n\, -(n-1) \cdots -2\, -1)$.  The map $\cmaj: W
\rightarrow C$ is given by
\[
\cmaj(w) = \begin{cases} \id & \mbox{if $w_n > 0$} \\ c & \mbox
{if $w_n < 0$} \end{cases}
\]
Theorem \ref{T:qWeyl} states in this case
$$
\sum_{w\in W} q^{\cdes(w)}\,e^{\cmaj(w)} = (e^{\id} + e^c)\cdot A_{n}(q)\cdot(1+q)^{n-1}.
$$

\section{Gr\"{o}bner bases}
\label{sec:grobner} In this section we will regularly refer to the
results of the first paper in this series~\cite{API}, in particular
Appendix 19.

Let $\Phi \subset V$ be a fixed irreducible root system and $P$ a
proper alcoved polytope.  We first note that the triangulations
of alcoved polytopes are coherent.
\begin{lemma}
\label{coherent} Any polytopal subdivision arising from a hyperplane
arrangement is coherent.
\end{lemma}

\begin{proof}
As only finitely many hyperplanes will be involved in a
triangulation or subdivision of a polytope we may assume the set
$\S$ of hyperplanes is finite.

Pick a linear functional $\phi_H$ for each hyperplane $H \in \S$
such that $H$ is given by the vanishing of $\phi_H$.  Then define
the piecewise linear convex function $h: V \rightarrow \R$ by
\[
h(v) = \sum_{H \in \S} |\phi_H(v)|.
\]
It is clear that $h(v)$ is convex as it is a sum of convex
functions.  The domains of linearity are exactly the regions
determined by the hyperplane arrangement.  Thus the subdivision of a
polytope induced by a hyperplane arrangement is coherent.
\end{proof}

We denote by $N$ the set of vertices of the
affine Coxeter arrangement.  By \cite[Theorem~19.1]{API} the triangulation of $P$ can be
described by some appropriate term order on the polynomial ring
\[
k[P] = k[x_a | a \in N \cap \P].
\]
Let us fix coordinates on $V$ so that all points in $N$ have integer
coordinates.  Identify a vertex $a = (a_1,\ldots,a_n) \in N$ with
the coordinates $(a,1) \in V \oplus \R$.  Thus the triangulation is
also equivalent to the reduced Gr\"{o}bner basis $\g_\P$ of the
toric ideal $I_{\P} : = I_{\P \cap N}$ in the notation of
\cite[Appendix 19]{API}. By our choice of coordinates this toric
ideal is homogeneous.

In general the Gr\"{o}bner basis $\g_\P$ appears to be quite
complicated but many simplifications occur when $N$ is a lattice.
One can check directly that this is the case for the root systems
$A_n$, $C_n$ and $D_4$.

We assume that $\Phi$ is one of these root systems from now on.  Set
$c_i = \frac{\omega_i}{a_i}$. Then $N$ is spanned by the $c_i$.  In
this case an alcove has normalized volume $1$ with respect to $N$.
Thus by \cite[Proposition~19.2]{API}, $\g_\P$ has an initial ideal
generated by square-free monomials.

\begin{example}
With the notation as in Section~\ref{sec:typeC}, the vertices $N$ of
the affine Coxeter arrangement of type $C_n$ are exactly the points with all
coordinates either integers or half-integers.  One can check that
the lattice $N$ is spanned by the vectors $c_i$.
\end{example}

\begin{lemma}
\label{lem:midpoint} Let $a, b \in N$.  The midpoint $(a+b)/2$ is
either a vertex or it lies on a unique edge, such that it is the
midpoint of the two closest vertices lying on that edge. 
\end{lemma}

\begin{proof}
Suppose $c  = (a+b)/2$ is not in $N$. The closed fundamental alcove
$\overline{A_\circ}$ is the convex hull of the points $c_i$ and $0$.
Thus there is a (affine) Weyl group element $\sigma$ which takes the
midpoint $c = (a+b)/2$ into $\overline{A_\circ}$.  Since $\sigma(a)$
and $\sigma(b)$ are both integral linear combinations of the $c_i$,
it is clear that $\sigma(c)$ must have the form $\frac{c_i +
c_j}{2}$ or $\frac{c_i}{2}$. In the first case, $\sigma(c)$ lies on
the edge given by the intersection of the hyperplanes
$H_{\alpha_k,0}$ for $k \neq i,j$ satisfying $k \in [1,r]$ and
$H_{\theta,1}$.  In the second case the edge is given by the
hyperplanes $H_{\alpha_k,0}$ for $k \neq i$.
Thus $c$ is the midpoint of $\sigma^{-1}(c_i)$ and
$\sigma^{-1}(c_j)$, or the midpoint of $\sigma^{-1}(c_i)$ and $\sigma^{-1}(0)$.
\end{proof}

In the first case of Lemma~\ref{lem:midpoint}, we set $u(a,b) = v(a,b) = (a+b)/2$.  In the second case
we set $u(a,b)$ and $v(a,b)$ to be the two closest vertices on the
edge containing $(a+b)/2$.

\begin{example}
For type $A_{n-1}$ we can describe the vertices $u(a,b)$ and $v(a,b)$ in the following explicit manner  (\cite{API}).  Let $I,J$ be two $k$-element multi-subsets of $[n]$.  Let $a_1 \leq a_2 \leq \cdots \leq a_{2k}$ be the increasing rearrangement of $I \cup J$.  We define two $k$-element multi-subsets $U(I,J)$ and $V(I,J)$ by $U(I,J) = \{a_1,a_3,\ldots,a_{2k-1}\}$ and $V(I,J) = \{a_2,a_4,\ldots,a_{2k}\}$.  For a $k$-element multi-subset $I$, we let $a_I \in
{\mathbb R}^n$ be the (integer) vector with $j$-th coordinate $(a_I)_j$ equal to the number of occurrences of $\{1,2,\ldots,j\}$ in $I$.  Then one can check that $u(a_I,a_J)$ and $v(a_I,a_J)$ are exactly $a_{U(I,J)}$ and $a_{V(I,J)}$.
\end{example}

\begin{lemma}\label{L:alcovedmidpoint}
Suppose $a, b \in \P$ are vertices of the affine Coxeter
arrangement, where $\P$ is a proper alcoved polytope. Then the
vertices $u(a,b)$ and $v(a,b)$ are also
in $\P$.
\end{lemma}
\begin{proof}
As $\P$ is convex, $c = (a+b)/2 \in \P$. Assume now that $c$ is not
a vertex and suppose $u(a,b) \notin \P$. Then there exists some root
$\alpha$ and some integer $k$ so that $H_{\alpha,k}$ separates
$u(a,b)$ and $c$.  Here we pick $H_{\alpha,k}$ so that it may go
through $c$ but not through $u(a,b)$.  The intersection of
$H_{\alpha,k}$ and the edge joining segment joining $u(a,b)$ to
$v(a,b)$ is a vertex of the affine Coxeter arrangement.  But this is
impossible, as there are no vertices lying between $v(a,b)$ and
$u(a,b)$.
\end{proof}

Define a marked set $\mathcal{G}_{\P}$ of elements which lie in
$I_{\P}$ as follows:
\begin{equation}
\label{eq:gp} {\mathcal G}_{\P} = \{ \underline{x_a x_b} -
x_{u(a,b)} x_{v(a,b)} \},
\end{equation}
where $a,b$ range over pairs of unequal vertices in $\P$.  The main
result of this section is the following theorem.

\begin{theorem}
\label{thm:grobtri} Let $\Phi$ be one of the root systems $A_n$,
$C_n$ or $D_4$ and $\P$ a proper alcoved polytope.  Then there exists a term order $\prec_\P$ such that the quadratic
binomials $\mathcal{G}_{\P}$ form a (reduced) Gr\"{o}bner basis of
the toric ideal $I_{\P}$ with respect to $\prec_\P$,
such that the underlined monomial is the leading term.
\end{theorem}

\begin{proof}
By Lemma~\ref{L:alcovedmidpoint}, the binomials in $\mathcal{G}_{\P}$ do
indeed make sense and since $a+b = u(a,b) + v(a,b)$, they lie within
$I_{\P}$.

By Lemma~\ref{coherent}, the triangulation is coherent and is given
by the domains of linearity of the piecewise-linear function $h$.
The same function $h$ gives a weight vector $\w$ as described in
\cite[Appendix]{API}. By \cite[Theorem~A.1]{API} the weight
vector $\w$ induces a term order $\prec_\P$ such that
$\D_{\prec_\P}(I_\P) = \D_\w$ (we have also used the fact that the
triangulation is unimodular, and \cite[Proposition~A.2]{API}).

Now let $a,b \in \P$ be vertices of the affine arrangement.  If
$x_ax_b \neq x_{u(a,b)}x_{v(a,b)}$ then clearly $a$ and $b$ do not
belong to the same simplex of the triangulation of $\P$.  Thus
$x_ax_b$ belongs to the Stanley-Reisner ideal of the alcoved
triangulation of $\P$ while $x_{u(a,b)}x_{v(a,b)}$ does not. This
implies that under $\prec_\P$, the underlined terms in the basis
above are exactly the leading terms.  In other words, the set
$\g_\P$ is `marked coherently'.

Finally we check that ${\mathcal G}_{\P}$ is indeed a Gr\"{o}bner
basis of $I_\P$ under $\prec$.  Since $\g_\P$ is marked coherently,
it follows that the reduction of any polynomial modulo $\g_\P$ is
Noetherian (that is, it terminates).  It is clear that a monomial
$x_{p_1} \ldots x_{p_k}$ cannot be reduced further under $\g_p$ if
any only if all the $p_i$ belong to the simplex of the
triangulation.  Thus every monomial can be reduced via $\g_\P$ to a
standard monomial and hence $\inn_\prec(\g_\P)$ generates
$\inn_\prec(I_\P)$.

The fact that this Gr\"{o}bner basis is reduced is clear.
\end{proof}

\begin{cor}
\label{cor:midconverse} Let $\Phi$ be one of the root systems $A_n$,
$C_n$ or $D_4$ and $\P \subset V$ be a convex polytope with vertices
amongst the vertices of the affine Coxeter arrangement. Then $\P$ is
alcoved if and only if the conclusion of Lemma~\ref{L:alcovedmidpoint}
holds.
\end{cor}

\begin{proof}
`Only if' is the content of Lemma~\ref{L:alcovedmidpoint}.  For the other
direction we note that the quadratic binomials $\g_\P$ can be
defined by (\ref{eq:gp}).  There is some large alcoved polytope $\Q$
which contains $\P$ and since $\g_\P \subset \g_\Q$ this allows us
to conclude that $\g_\P$ is marked coherently.  And so there is a
term order $\prec_\P$ which selects the marked monomial as leading
monomial in $\g_\P$.  It is easy to check that $\g_\P$ is the
Gro\"{o}bner basis of $I_\P$ with respect to $\prec_\P$, and the
standard monomials are exactly given by monomials corresponding to
faces of alcoves.  Thus we obtain an alcoved triangulation of $P$.
\end{proof}


%

Naturally associated to the ideal $I_{\P}$ is the projective
algebraic variety $Y_{\P}$ defined as
\[
Y_{P} = \text{Proj} \left ( k[x_a | a \in {\mathcal P}]/I_{\P}
\right ).
\]
This is the projective toric variety associated to the polytope
${\P}$.  The following corollary is immediate from Theorem~\ref{thm:grobtri}
and~\cite[Proposition~A.2]{API}. 

\begin{cor}
Let $\Phi$ be one of the root systems $A_n$, $C_n$ or $D_4$ and $\P$
a proper alcoved polytope. Let $Y_\P$ be the projective toric
variety defined by $I_\P$.  Then $Y_\P$ is projectively normal and
its Hilbert polynomial equals to the Erhart polynomial of $\P$ (with
respect to $N$).
\end{cor}

This should be compared with the work of Payne \cite{Pay}, who showed, in types $A$, $B$, $C$, and $D$, that alcoved polytopes whose vertices lie in the coweight lattice are normal with respect to the coweight lattice.

\end{document}